\documentclass{amsart}
\usepackage[english]{babel}
\usepackage{amsmath,amssymb,amsthm}
\usepackage{bbm}
\usepackage{color}

\newtheorem{thm}{Theorem}[section]

\newtheorem{cor}[thm]{Corollary}
\newtheorem{prop}[thm]{Proposition}

\begin{document}

\title[]
{Interpolation theorems for operators in non-commutative
vector-valued symmetric spaces}

\author{V.I.Chilin}
\address{V.I.Chilin\\
 Department of algebra and functional analisis  \\
Faculty of mathematics and mechanics, Uzbek National University \\
Tashkent, 100174, Uzbekistan} \email{\tt chilin@@ucd.uz;\tt
vladimirchil@@gmail.com}

\author{A.K.Karimov}
\address{A.K.Karimov\\
Department of Mathematics\\
Tashkent Institute of Textile and Light Industry\\
Tashkent, 100100, Uzbekistan} \email{\tt
abdusalom.karim@@gmail.com;\tt karimov57@@rambler.ru}

\begin{abstract}
We prove the version of interpolation theorem for
non-commutative vector-valued fully symmetric spaces associated
with fully symmetric Banach function spaces and a von Neumann
algebra equipped with a faithful semifinite normal trace.

\end{abstract}
\maketitle

\section{Introduction.}
The development of the theory of non-commutative integration, initiated by I.E. Segal \cite{Se}, allowed to construct the theory of symmetric spaces of measurable operators
similar to the theory of symmetric Banach function spaces.
Naturally that a general theory of interpolation of linear operators
for symmetric Banach function spaces (see e.g. \cite{BL},\cite{KPS}) got its 'noncommutative' development.   Non-commutative variant of
Riesz-Thorin theorem firstly was proved by Kunze \cite{K}, while
an application of the real interpolation method in the
non-commutative setting was given in \cite{PS}. A general approach
to construct an interpolation theory in the case of trace ideals
has been explored by Arasy \cite {A}. In the case of arbitrary von
Neumann algebras this approach was developed by
Dodds-Dodds-Paghter \cite{DDP}.

On the other hand the development of non-commutative ergodic
theory leads to study of special classes of non-commutative vector
valued $L_p$-spaces $L_p(M,\tau,I)$ associated with a von-Neumann
algebra $M,$ faithful normal semifinite trace $\tau$ on $M$ and
with an indexed set $I$ ( \cite {P}). In particular, it was
necessary to prove various variants of interpolation theorems for
$L_p(M,\tau,I)$.

In this paper, it is defined a new class of non-commutative fully
symmetric spaces $E(M,\tau,I),$ associated with $(M,\tau,I)$ and
with fully symmetric Banach function space $E([0,\infty))$ of
measurable functions defined on $[0,\infty).$ We prove that the
Banach space $E(M,\tau,I)$ is isometric to complemented subspace
in non-commutative symmetric spaces $E\left( {M \otimes B\left(
{l_2 \left( I \right)} \right)} \right)$, where $B\left( {l_2
\left( I \right)} \right)$  is  the algebra of all bounded linear
operators on the complex Hilbert spaces $l_2 \left( I \right).$
 By means of this isometry  and
non-commutative interpolation
 theorem in (\cite {DDP})  we are able to prove that the next
 version of interpolation theorem for  spaces $E(M,\tau,I).$

\begin{thm}
\label{t_1_1}
Let $E,F,E_1,E_2,F_1,F_2$ be fully symmetric Banach
function spaces on $[0, + \infty),$ let $M_i$ be a von Neumann
algebra, let $\tau_i$ be a faithful normal semifinite trace on
$M_i, i=1,2,$ and let $I$ be an arbitrary indexed set.

(i) If $(E,F)$ is an (exact) interpolation pair for the pair
$((E_1,F_1),(E_2,F_2))$ then $(E(M_1,\tau_1,I),F(M_2,\tau_2,I))$is
an (exact) interpolation pair for
$$((E_1(M_1,\tau_1,I),F_1(M_1,\tau_1,I)),(E_2(M_2,\tau_2,I),F_2(M_2,\tau_2,I))).$$
(ii) If $(E,F)$ is an exact interpolation pair of exponent
$\theta\in[0,1]$ for the pair $((E_1,F_1),(E_2,F_2)),$ then
$(E(M_1,\tau_1,I),F(M_2,\tau_2,I))$ is an exact interpolation
pair of exponent $\theta\in[0,1]$ for
$$((E_1(M_1,\tau_1,I),F_1(M_1,\tau_1,I)),
(E_2(M_2,\tau_2,I),F_2(M_2,\tau_2,I))  ).$$

\end{thm}

We use terminology and notations from the theory of von Neumann algebras
(\cite {SZ}, \cite {Ta}) and the theory of measurable
operators from (\cite {MCh},\cite {Se}).

\section{Preliminaries}

Let  $H$ be a Hilbert space over the field $\mathbb{C}$ of complex
numbers, let $B\left( H \right)$ be the $ * $ - algebra of all
bounded linear operators on $H$, let $\textbf{1}$ be the identity
operator on $H$ and let $M$ be a von Neumann subalgebra of
$B\left( H \right)$. By $P\left( M \right) = \left\{ {p \in M:p^2
= p = p^ * } \right\}$ we denote the lattice of all projections in
$M$ and by $P_{fin} \left( M \right)$ the sublattice of its finite
projections.

A closed linear operator $x$ affiliated with a von Neumann algebra
$M$ with  dense domain $D\left( x \right) \subset H$ is
called $measurable$ if there exists a sequence $\left\{ {p_n }
\right\}_{n = 1}^\infty   \subset P\left( M \right)$ such that
$p_n  \uparrow \textbf{1}, p_n \left( H \right) \subset D\left( x
\right)$ and $p_n ^ \bot   = \textbf{1} - p_n \in P_{fin} \left(M
\right)$ for each $n\in \mathbb{N},$ where $\mathbb{N}$ is the set
of all natural numbers.

The set $S\left( M \right)$ of all measurable operators is a $ * -
$ algebra with identity $\textbf{1}$ over the field $\mathbb{C},$
in addition  $M$ is a $*-$ subalgebra of $S\left( M \right)$
\cite{Se}.

For every subset $E \subset S\left( M \right),$ the set of all
selfadjoint (resp., positive) operators in $E$ is denoted
by $E_h $ (resp,. $E_ +  $ ). The partial order in $S_h (M)$ defined by its cone $S_+(M)$ is denoted by
$\leq$.  For a net $\{x_\alpha\}_{\alpha\in A}\subset S_h (M)$,
the notation $x_\alpha \uparrow x$ (resp., $x_\alpha \downarrow
x$), where $x \in S_h \left( M \right),$ means that $x_\alpha \le
x_\beta  $ (resp., $x_\beta \le x_\alpha $ ) for  $\alpha \le
\beta $ and  $x = $ $\mathop {\sup }\limits_{\alpha \in A}
x_\alpha $ (resp., $x = $ $\mathop {\inf }\limits_{\alpha \in A}
x_\alpha $).

Let $x$ be closed linear operator with dense domain $D\left( x
\right)$ in $H$, and    let   $x = u\left| x \right|$ be the polar
decomposition of the operator $x,$ where $\left| x \right| =
\left( {x^ *  x} \right)^{{1 \over 2}} $ and $u$ is the partial
isometry in $B\left( H \right)$ such that $u^ * u$ is the right
support of $x$. It is known that $x \in S\left( M \right)$ if and
only if $\left| x \right| \in S\left( M \right)$ and $u \in
 M $. If $x$ is self-adjoint operator affiliated with
$M,$ then the spectral family  of projections $\left\{ {E_\lambda
\left( x \right)} \right\}_{\lambda  \in \mathbb{R}} $ for $x$
belongs to $M$, where $\mathbb{R}$ is the set of real numbers.

Let $M$ be a von Neumann algebra with a faithful semifinite normal
trace $\tau $. A densely-defined closed linear operator $x$
affiliated with $M$ is said to be $\tau $ -measurable if for each
$\varepsilon  > 0$ there exists a projection $p \in  P\left( M
\right)$  such that $p\left( H \right) \subseteq D\left( x
\right)$ and $\tau \left( {\textbf{1} - p} \right) \le
\varepsilon$. Let $S \left( {M,\tau } \right)$ be the set of
all $\tau $ -measurable operators. It is clear that $S
\left( {M,\tau } \right)$ is $ * $ - subalgebra in $ S\left( M
\right)$ and $M \subset S\left( {M,\tau } \right)$.

For every $x \in  S \left( {M,\tau } \right)$ we define its
generalized singular numbers by $\mu_t \left( {x} \right) = \inf
\left\{ {s > 0:\tau \left( \textbf{1}-E_s(x) \right) \le t}
\right\}, t>0 $. Let $V\left( {\varepsilon ,\delta } \right) =
\left\{ {x \in S \left( {M,\tau } \right):\mu_\delta \left( {x}
\right) \le \varepsilon } \right\}$. It is known that in $S \left(
{M,\tau } \right)$ there exists Hausdorff vector topology  $t_\tau
$  with base of neighborhoods of zero given by $\left\{ {V\left(
{\varepsilon ,\delta } \right):\varepsilon ,\delta  > 0}
\right\}$. This topology $t_\tau $  is called the topology of convergence
in measure or measure topology. The pair $\left ( S \left( {M,\tau
} \right), t_\tau \right )$ is a complete metrizable topological $
* $ - algebra and $M$ is dense in $\left ( S \left( {M,\tau }
\right), t_\tau \right )$ \cite{N}.

We need the following property of the measure topology $t_\tau$.

\begin{thm} \cite {Ti}. \label{t_2_1} Let $f : \mathbb{R} \to \mathbb{R}$ be a continuous function,
$x_n, x \in S_h \left( {M, \tau } \right)$ and $x_n
\buildrel {t_\tau  } \over \longrightarrow x$. Then
  $f\left( {x_n } \right)\buildrel
{t_\tau  } \over \longrightarrow f\left( x\ \right)$.
\end{thm}
A Banach space $(E,\left\| \cdot \right\|_E  )$ which is a linear
subspace of $S \left( {M,\tau } \right)$  is called fully
symmetric if conditions $x\in E,$ $y\in S \left( {M,\tau }
\right),$  $\int\limits_{{0}}^s  \mu_t\left( {y}
\right)dt\leq\int\limits_{{0}}^s  \mu_t\left( {x} \right)dt$ for
all $s>0$ imply that $y\in E$ and $\left\| y \right\|_E\leq\left\|
x \right\|_E.$ The space $(E,\left\| \cdot \right\|_E  )$ is said
to have the Fatou property if the conditions  $0\leq x_\alpha \in E,
x_\alpha \leq x_\beta $ for $\alpha\leq \beta,$   $\mathop {\sup
}\limits_{\alpha } \|x_\alpha \|_E<\infty$ imply that there exists $x =
$ $\mathop {\sup }\limits_{\alpha } x_\alpha $ in $E$ and $\left\|
x \right\|_E =\mathop {\sup }\limits_{\alpha } \|x_\alpha \|_E.$
The space $(E,\left\| \cdot \right\|_E  )$ is said to have order
continuous norm if $x_\alpha\in E,$ $x_\alpha\downarrow0$ implies
$\left\| x_\alpha \right\|_E\downarrow0.$ It is shown in
\cite{DDP} that if $(E,\left\| \cdot \right\|_E  )$ is a fully
symmetric space, $x_n,x\in E$ and $\left\| x_n
-x\right\|_E\rightarrow0$ then $x_n \buildrel {t_\tau  } \over
\longrightarrow x.$

Let $L^0{([0,\infty))}$ be the linear space of all (equivalance
classes of) almost everywhere finite complex-valued Lebesque
measurable functions on the half line $[0,\infty).$  We identify
$L^\infty {([0, \infty))}$ with the commutative von Neumann algebra
acting by multiplication on the Hilbert space $L^2{([0,\infty))}$ equipped
with  trace given by integration with respect to the Lebesque measure.
A Banach space $E([0,\infty))\subset L^0{([0,\infty))}$ is called
a fully symmetric Banach function space on $[0,\infty)$ if the
corresponding condition above holds with respect to the von Neumann
algebra $L^\infty {([0, \infty))}.$

Let $E=E([0,\infty))$ be a fully symmetric Banach function space
on $[0,\infty).$  We define $E(M,\tau)=E(M)=\left\{ x\in
S(M,\tau):\mu_t{(x)}\in E \right\}$ and set $\left\| x
\right\|_{E(M)}=\left\| \mu_t {(x)} \right\|_E , x\in E(M).$ It is
show in \cite{DDP} that $E(M)$ is a fully symmetric space in
$S(M,\tau).$ If $p\geq1$ and $E([0,\infty))=L_p{([0,\infty))}$,
then
$$L_p{(M,\tau)}=\left\{ x\in S(M,\tau):\left\| x \right\|_p=
(\int\limits_{0}^\infty \mu_t^{p}(\left| x \right|)
dt)^{1/p}=(\int\limits_{0}^\infty \lambda^p d \tau(E_\lambda
(\left| x \right|)))^{1/p}<\infty\right\}$$ is a non-commutative
$L_p-$ space with respect to the order continuous norm $\left\|
\cdot \right\|_p$ \cite{Y}.

Let $H_1 \otimes H_2$ be the tensor product of Hilbert spaces of
$H_1$ and $H_2,$  let $M_1$ and $M_2$ be von Neumann algebras
acting in Hilbert spaces $H_1$ and $H_2$ respectively. Denote by
$M_1 \otimes M_2$ the tensor product of the von Neumann algebras
$M_1$ and $M_2$, i.e. the von Neumann algebra in $B\left( {H_1
\otimes H_2} \right)$ generated by $\ast-$ algebra
 $ \left\{ {\sum\limits_{k = 1}^n {x_k
\otimes y_k; x_k \in M_1}, y_k \in M_2, n \in \ \mathbb{N}}
\right\}.$

 We need the following properties of tensor product of von Neumann
 algebras.

\begin{thm}\cite{Ta}. \label{t_2_2}
 Let $\tau _i $    be a faithful normal
semifinite trace on von Neumann algebra $M_i$, $i = 1,2$. Then
there exists unique faithful normal semifinite trace $\rho $ on
von Neumann algebra $M_1  \otimes M_2$ such that
$$\rho \left( {x \otimes y} \right) = \tau _1 \left( x
\right)\tau _2 \left( y \right)  \text{for every} \ x \in
(M_1)_{+} , \ y \in (M_2)_{+} .$$
\end{thm}

The trace $\rho $ is called tensor product of traces $\tau _1 $ and
$\tau _2 $ and is denoted by $\rho  = \tau _1  \otimes \tau _2 $. If
$x$ and $y$ are densely defined closed linear operators on $H_1 $
and $H_2 $ respectively, we define their algebraic tensor product,
denoted by $x \otimes y$, as the closure of the smallest linear
extension of the map $\xi \otimes \zeta \to x\xi  \otimes y\zeta$
where $\xi \in D\left( {x} \right)$ \ $\zeta  \in D\left( {y }
\right)$. This closure is correctly defined and  $\left( {x \otimes y}
\right)^* = x^*   \otimes y^*$ \cite{SZ}.

It is known \cite{St}  that if  $ x_i \in  L_1 \left(
{M_i,\tau_i }\right) $,  $ i=1,2, $ then   $x_1 \otimes
x_2 \in L_1 \left( {M_1 \otimes M_2},\tau_1\otimes\tau_2 \right),$
in addition $\left\| {x \otimes y} \right\|_1 = \left\| x
\right\|_1 \left\| y \right\|_1 .$ Since   $E([0,\infty))\subset
L_1{([0,\infty))}+L_\infty{([0,\infty))}$ for every fully
symmetric Banach function space on $[0,\infty)$ [\cite{KPS},
ch.II, $\S$4], it follows that $x_1 \otimes x_2 \in L_1 \left( {M_1 \otimes
M_2},\tau_1\otimes\tau_2 \right)+{M_1 \otimes M_2}$ for all
$x_i\in E(M_i),i=1,2.$

 Let $I$ be an arbitrary indexed set, and let $l_2 \left( I
\right)$ be the Hilbert space of all families $\left( {\alpha _i }
\right)_{i \in {\rm I}} $ of complex numbers with
$\sum\limits_{i \in {\rm I}} {\left| {\alpha _i } \right|^2 } <
\infty$. The scalar product in $l_2 \left( I \right)$ is defined by $\left( {\left( {\alpha _i } \right),\left( {\beta _i }
\right)} \right) = \sum\limits_{i \in I} {\alpha _i \overline
\beta  _i }$. For $j \in I$ we set $e_j  = \left( {\alpha_i^{j} }
\right)_{i \in I} \in l_2 \left( I \right)$, where $\alpha_i^{j}  = 0$
if $i \ne j$ and $\alpha_i^{i} = 1$.  For any  $\xi  =
\left( {\alpha_i } \right)_{i\in I} \in l_2 \left( I \right)$ we
have $\xi = \sum\limits_{i \in I} {\alpha_i e_i }$.  By $u_{ij} $
we denote the matrix units in $B\left( {l_2 \left( I \right)}
\right),$ i.e. $u_{ij} \left( { \xi} \right) = \left( {\xi, e_j }
\right)e_i $ for all $\xi \in B(l_2 \left( I
\right))$, $i,j \in I$.

We need the following well-known property of matrix units.

\begin{prop}
\label{pr_2_3}
If $x\in B(l_2(I))$,  then $p_ixp_j= (xe_j,e_i)u_{ij}$ for all $i,j\in I$, where $p_i = u_{ii}$.

\end{prop}
Let $M$ be an arbitrary von Neumann algebra and let $\tau$ be a faithful
normal semifinite trace on $M.$ We denote by $tr$ the canonical
trace on $B(l_2(I))$ and consider the tensor product $M \otimes
B(l_2(I) ).$ Let $E([0,\infty))$ be a fully symmetric Banach
function space on $[0,\infty).$

\begin{prop}
\label{pr_2_4}
If $x\in E(M),$ $i,j\in I,$ then $x\otimes u_{ij}\in E(M \otimes B(l_2(I))),$
in addition  $\left\| x \right\|_{E(M)}=\left\| x\otimes u_{ij}
\right\|_{E(M \otimes B(l_2(I)))}.$
\end{prop}

\begin{proof} Since $\left( {x \otimes u_{ij}} \right)^*\left( {x \otimes u_{ij}} \right)
= x^*x   \otimes u_{jj},$ it follows that $$|x \otimes u_{ij}|=\left( ({x
\otimes u_{ij}} ) ^* \left( {x \otimes u_{ij}}\right)
\right)^{1/2}=\left(|x|^2\otimes p_j\right)^{1/2}=|x| \otimes
p_j.$$ Thus, $x \otimes u_{ij} \in L_1(M \otimes B(l_2(I)))+M
\otimes B(l_2(I)),$ in addition,  $$\mu_t(x \otimes
u_{ij})=\mu_t(|x| \otimes p_j)=\inf \{s>0:(\tau \otimes
tr)(\textbf{1}_M\otimes \textbf{1}_{B(l_2(I))}-E_s(|x|\otimes
p_j))\leq t \}.$$  Using the functional calculus for positive
measurable operators [\cite{MCh}, ch.2,$\S\S$ 2-3] and the
equalities $|u_j|=p_j,$ $tr(p_j)=1,$ $(|x|\otimes
\textbf{1}_{B(l_2(I))})(\textbf{1}_M \otimes p_j)=|x|\otimes p_j =
(\textbf{1}_M\otimes p_j)(|x|\otimes \textbf{1}_{B(l_2(I))})$   we
have that $\mu_t(|x| \otimes
p_j)=\mu_t(|x|)tr(p_j)=\mu_t(x)\in E([0,\infty))$. Hence $x
\otimes u_{ij} \in E(M\otimes B(l_2(I)))$
 and $\|x\otimes u_{ij}\|_{E(M\otimes B(l_2(I)))}=\|\mu_t(x \otimes
u_{ij})\|_{E([0,\infty))}=\|\mu_t(x)\|_{E([0,\infty))}=\|x\|_{E(M)}.$
\end{proof}

 Following \cite{KPS}, a Banach couple (X,Y) is a pair of Banach spaces
$(X,\|\cdot\|_X),(Y,\|\cdot\|_Y)$ which are algebraically and
topologically embedded in a Hausdorff topological vector space.
With any Banach couple (X,Y) the following Banach
spaces are associated:

(i) the space $ X\cap Y $ equipped with the norm $\|x\|_{X\cap
Y}=\max\{\|x\|_X,\|x\|_Y\},$ $x\in X\cap Y;$

(ii) the space $X+Y$  equipped with the norm    $\|x\|_{X+
Y}=\inf\{\|y\|_X+\|z\|_Y $ $:x=y+z,  y\in X, z \in Y \},$
 $x\in X+Y.$

Let $(X_1,Y_1)$ and  $(X_2,Y_2)$ be Banach couples. A linear map
$T$ from the space $X_1+Y_1$ to the space $X_2+Y_2$ is called a
bounded operator from the couple $(X_1,Y_1)$ to the couple
$(X_2,Y_2)$ if $T$ is bounded operator from $X_1$ into
$X_2$  and $Y_1$  into  $Y_2,$ respectively.

Denote by $B((X_1,Y_1),(X_2,Y_2))$ the linear space of all bounded
linear operators from the couple $(X_1,Y_1)$ to the couple
$(X_2,Y_2).$ This space is a Banach space equipped with the norm
$$\|T\|_{B((X_1,Y_1),(X_2,Y_2))}=\max(\|T\|_{X_1\rightarrow
Y_1},\|T\|_{X_2\rightarrow Y_2}).$$
A Banach space $(Z,\|\cdot\|_Z)$ is said to be intermediate for
the Banach couple $(X,Y)$ if $X \cap Y \subset Z \subset X+Y$ with
continuous inclusions. If $(X_1,Y_1)$ and  $(X_2,Y_2)$ are two
Banach couples, and $Z_1,Z_2$ are  intermediate spaces for the
couple $(X_1,Y_1),$ $(X_2,Y_2)$ respectively, then the pair
$(Z_1,Z_2)$ is called interpolation pair for the pair $((X_1,Y_1),
(X_2,Y_2))$ if every bounded operator from the couple
$(X_1,Y_1)$ to the couple  $(X_2, Y_2)$  acts boundedly from $Z_1$
to $Z_2.$

If $(Z_1,Z_2)$ is an interpolation pair for the pair
$((X_1,Y_1),(X_2,Y_2)),$ then there exists a constant $c>0$ such
that $\|T\|_{Z_1\rightarrow Z_2}\leq
c\|T\|_{B((X_1,Y_1),(X_2,Y_2))}$ for all $T\in
B((X_1,Y_1),(X_2,Y_2)).$

An interpolation pair $(Z_1,Z_2)$ for the pair
$((X_1,Y_1),(X_2,Y_2))$ of Banach couples is called an exact
interpolation pair (resp. exact interpolation pair of exponent
$\theta\in [0,1]$) for the pair $((X_1,Y_1),(X_2,Y_2))$) if
$\|T\|_{Z_1\rightarrow Z_2}\leq \|T\|_{B((X_1,Y_1),(X_2,Y_2))}$
(resp. $\|T\|_{Z_1\rightarrow Z_2}\leq \|T\|_{X_1\rightarrow
Y_1}^{1-\theta}\|T\|_{X_2\rightarrow Y_2}^\theta $) for all $T\in
B((X_1,Y_1),(X_2,Y_2)).$

We need  the following non-commutative interpolation theorem for
spaces $E\left( {M,\tau } \right)$.

\begin{thm}(\cite{DDP}).
\label{t_2_5}
Let $E,F,E_1,E_2,F_1,F_2$ be fully symmetric
Banach function spaces on $[0,+\infty),$ let $M_i$ be a von Numann
algebra and let $\tau_i$ be a faithful semifinite normal trace on
$M_i,i=1,2.$

(i) If $(E,F)$ is an (exact) interpolation pair for the pair
$((E_1,F_1),(E_2,F_2))$ then $(E(M_1),F(M_2))$ is an (exact)
interpolation pair for the pair
$((E_1(M_1),F_1(M_1))$, $(E_2(M_2),F_2(M_2)));$

(ii) If $(E,F)$ is an exact interpolation pair of exponent
$\theta\in [0,1]$ for the pair $((E_1,F_1),(E_2,F_2))$
then $(E(M_1),F(M_2))$ is an exact interpolation pair of exponent
$\theta\in [0,1]$  for the pair
$((E_1(M_1),F_1(M_1)),(E_2(M_2),F_2(M_2)))$.
\end{thm}

\section{Non-commutative vector valued fully symmetric spaces}

In this section we introduce a class of non-commutative
vector-valued fully symmetric spaces $E(M,\tau,I)$ associated with
von Neumann algebra $M$ equipped with a faithful normal semifinite trace
$\tau$ and with an arbitrary indexed set $I.$ Moreover, we construct
 an isometric linear map from  $E(M,\tau,I)$ into
 $E\left( {M \otimes B\left( {l_2 \left( I \right)} \right)} \right).$

Let $I$ be an arbitrary indexed set and let $\Gamma$ be a directed
set of all finite subsets of $I$ ordered by inclusion. Let
$E=E((0,\infty])$ be a fully symmetric Banach function space on
$[0,\infty).$  Let $M$ be a von Neumann algebra with a faithful
semifinite normal trace $\tau.$ Define a linear space $E_0 \left(
{M, \tau, I} \right)$ as follows:
$$E_0\left( {M, \tau, I} \right)= \{ \left\{ {a_i } \right\}_{i \in \gamma }: a_i\in E \left(
{M,\tau } \right),  i\in \gamma  \in \Gamma \};$$
$$ \lambda \left\{ {a_i } \right\}_{i\in \gamma}=\left\{
{\lambda a_i } \right\}_{i\in \gamma},\lambda \in \mathbb{C};$$
$$\left\{ {a_i} \right\}_{i\in\alpha} +\left\{ { b_i }
\right\}_{i\in\beta}= \left\{ {a_i + b_i }
\right\}_{i\in\alpha\cup\beta},$$
where $a_i=0$ for $i\in
\beta\backslash\alpha$ and $b_i=0$ for
$i\in\alpha\backslash\beta.$

Denote by $\sigma_s$ the bounded linear operator in $E[0,\infty)$
for which $(\sigma_{s}f)(t)=f(s^{-1}t),$ $s>0,$ $t>0,$ $f\in
E([0,\infty))$ [KPS, Ch.II, $\S$4]. According to \cite{FK} for any
$x,y\in E(M,\tau),t>0$ we have $\mu_{t}(x+y)\leq
\mu_{t/2}(x)+\mu_{t/2}(y),$ i.e. $\mu_{t}(x+y)\leq
\sigma_{2}(\mu_{t}(x))+ \sigma_{2}(\mu_{t}(y)).$ Then for any
finite set $\{a\}_{i=1}^{n}\subset E(M,\tau),$ $n\in N,$ the
inequalities $$\mu_{t}(\sum\limits_{i =
1}^n(|a_i|^2))\leq\sum\limits_{i =
1}^{n}\sigma_{2^i}(\mu_{t}(|a_i|^2))= \sum \limits_{i =
1}^{n}(\sigma_{2^i}(\mu_{t}(|a_i|)))^2\leq (\sum \limits_{i =
1}^{n}\sigma_{2^i}(\mu_{t}(|a_i|)))^2$$
hold. Therefore,
$\mu_{t}(\sum\limits_{i = 1}^n(|a_i|^2)^{1/2})=
(\mu_{t}(\sum\limits_{i = 1}^n|a_i|^2))^{1/2}\leq \sum \limits_{i
= 1}^{n}\sigma_{2^i}(\mu_{t}(|a_i|)).$ Since
$\sigma_{2^i}(\mu_{t}(|a_i|))\in E([0,\infty))$ for any
$i=1,2,...n$, we have  $(\sum\limits_{i = 1}^n|a_i|^2)^{1/2}\in
E(M,\tau).$ Thus, for any $a = \left\{ {a_i } \right\}_{i \in
\gamma} \in E_0 \left( {M, \tau, I} \right)$ it is defined the
number
$$\left\| a \right\|_0  = \left\|{\left( {\sum\limits_{i \in
\gamma }^{} {\left| {a_i }\right|^2 } } \right)^{1/2} }
\right\|_{E(M,\tau)}.$$
\begin{prop}
\label{pr_3_1}
The function $\left\| \cdot \right\|_0: E_0\left( {M, \tau, I} \right)\rightarrow
\mathbb{R}$ is a norm on the linear space $E_0 \left( {M, \tau, I}
\right).$
\end{prop}
\begin{proof} Fix an index  $j_0 \in I$ and consider a linear operator $T:
E_0 \left( {M, \tau,I} \right) \to E \left( {M \otimes B\left(
{l_2 \left( I \right)} \right)} \right)$ defined by $T\left( a
\right) = \sum\limits_{i \in \gamma}^{} {a_i \otimes u_{ij_0 } }$,
where $a = \left\{ {a_i } \right\}_{i \in \gamma } \in  E_0 \left(
{M, \tau,I} \right)$. By Proposition \ref{pr_2_4} we have
that $T\left( a \right) \in E \left( {M \otimes B\left( {l_2
\left( I \right)} \right)} \right)$. Since
\begin{eqnarray*} |T ( a ) |^2  &=& T^ *
\left( a \right)T\left( a \right) = \left( {\sum\limits_{i \in
\gamma } {a_i  \otimes u_{ij_0 } } } \right)^
*  \left( {\sum\limits_{i \in \gamma } {a_i  \otimes u_{ij_0 } } }
\right) \\
&=& \left( {\sum\limits_{i \in \gamma } {a_i^* \otimes u_{j_0 i} }
} \right)\left( {\sum\limits_{i \in \gamma } {a_i \otimes u_{ij_0
} } } \right) \\
&=& \sum\limits_{i \in \gamma } {a_i^
*  a_i  \otimes } u_{j_0 j_0}  = \left( {\sum\limits_{i \in \gamma
} {\left| {a_i }  \right|}^2 } \right) \otimes p_{j_0}
\end{eqnarray*}  we have $|T(a)|= {\left(
{\sum\limits_{i \in \gamma }^{} {\left| {a_i} \right|^2 } }
\right)^{1/2} \otimes p_{j_0} }.$  Consequently, by proposition
\ref{pr_2_4},
$$\left\| {T\left( a \right)} \right\|_{E \left( {M \otimes
B\left( {l_2 \left( I \right)} \right)} \right)}= \|{\left(
{\sum\limits_{i \in \gamma }^{} {\left| {a_i} \right|^2 } }
\right)^{1/2} \otimes p_{j_0} }\|_{{E \left( {M \otimes B\left(
{l_2 \left( I \right)} \right)} \right)}}=$$
$$\|{\left(
{\sum\limits_{i \in \gamma }^{} {\left| {a_i} \right|^2 } }
\right)^{1/2} }\|_{E(M,\tau)}=\|a\|_0 .$$
If $\|a\|_0 =0$, then
$a_i=0$ for every $i\in\gamma,$ i.e. $a=0.$ Furthermore, $\|\lambda
a\|_0 = \left\| {T\left(\lambda a \right)} \right\|_{E \left( {M
\otimes B\left( {l_2 \left( I \right)} \right)} \right)} =
|\lambda|\left\| {T\left( a \right)} \right\|_{E \left( {M \otimes
B\left( {l_2 \left( I \right)} \right)} \right)}=|\lambda| \|
a\|_0$  for all  $a  \in E_0 \left( {M, \tau,I} \right),$ $\lambda
\in \mathbb{C}.$ Similarly, for any  $a,b  \in E_0 \left( {M,
\tau,I} \right)$ we have
$$\| a+b \|_0 = \left\| {T\left( a+b
\right)} \right\|_{E \left( {M \otimes B\left( {l_2 \left( I
\right)} \right)} \right)} \leq $$
$$  \leq  \left\| {T\left( a \right)}
\right\|_{E \left( {M \otimes B\left( {l_2 \left( I \right)}
\right)} \right)} +                  \left\| {T\left( b \right)}
\right\|_{E \left( {M \otimes B\left( {l_2 \left( I \right)}
\right)} \right)}=\|a\|_0 + \|b\|_0.$$
\end{proof}

 Denote by $E(M,\tau,I)$  the set of all $\{a_i\}_{i\in I}\subset
E(M,\lambda)$ with $ \|\{a_i\}_{i\in
I}\|_{E(M,\tau,I)}:=\mathop {\sup }\limits_{\gamma  \in \Gamma}
\left\| {\left( {\sum\limits_{i \in \gamma }^{} {\left| {a_i }
\right|^2 } } \right)^{1/2} } \right\|_{E(M,\tau)}<\infty.$ Let us
define algebraic operations on  $E(M,\tau,I)$  by the assumption
$$\{a_i\}_{i\in I} + \{b_i\}_{i\in I}= \{a_i + b_i\}_{i\in I},$$ $$\\
\lambda\{a_i\}_{i\in I}=  \{\lambda a_i\}_{i\in I}, \lambda
\in\mathbb{C}.$$
\begin{prop}
\label{pr_3_2}
$E(M,\tau,I)$ is a normed space.
\end{prop}
\begin{proof}It is clear that  $\lambda\{a_i\}_{i\in I}\in
E(M,\tau,I)$ and 
$$ \|\lambda\{ a_i\}_{i\in I}\|_{E(M,\tau,I)}=
|\lambda|\|\{a_i\}_{i\in I}\|_{E(M,\tau,I)}$$
for all
 $\{a_i\}_{i\in I}\in E(M,\tau,I), \lambda \in
\mathbb{C}.$  Let $\{a_i\}_{i\in I}, \{b_i\}_{i\in I} \ \in
E(M,\tau,I).$ For every $\gamma\in\Gamma$ we have

 $$\left\| {\left( {\sum\limits_{i \in \gamma }^{} {\left| {a_i  +
b_i } \right|^2 } } \right)^{1/2} } \right\|_{E(M,\tau)}  =
\left\| {\left\{ {a_i  + b_i } \right\}_{i \in \gamma}^{} }
\right\|_0 \leq \left\| {\left\{ {a_i } \right\}_{i \in \gamma }^{}
} \right\|_0 + \left\| {\left\{ {b_i } \right\}_{i \in
\gamma }^{} } \right\|_0 =$$
$$= \left\| {\left( {\sum\limits_{i \in \gamma }^{} {\left| {a_i }
\right|^2 } } \right)^{1/2} } \right\|_{E(M,\tau)}  + \left\|
{\left( {\sum\limits_{i \in \gamma }^{} {\left| {b_i } \right|^2 }
} \right)^{1/2} } \right\|_{E(M,\tau)} \leq $$
$$ \leq \mathop {\sup }\limits_{\gamma \in \Gamma} \left\| {\left({\sum\limits_{i \in \gamma }^{} {\left| {a_i } \right|^2 } }
\right)^{1/2} } \right\|_{E(M,\tau)}  + \mathop {\sup
}\limits_{\gamma \in \Gamma} \left\| {\left( {\sum\limits_{i \in
\gamma }^{} {\left| {b_i } \right|^2 } } \right)^{1/2}}
\right\|_{E(M,\tau)}.$$

 Consequently,
 $$\mathop {\sup }\limits_{\gamma  \in \Gamma}
\left\| {\left( {\sum\limits_{i \in \gamma }^{} {\left| {a_i  +
b_i } \right|^2 } } \right)^{1/2} } \right\|_{E(M,\tau)}  \le
\mathop {\sup }\limits_{\gamma \in\Gamma} \left\| {\left(
{\sum\limits_{i \in \gamma }^{} {\left| {a_i } \right|^2 } }
\right)^{1/2} } \right\|_{E(M,\tau)}  +$$
$$+ \mathop {\sup
}\limits_{\gamma \in \Gamma} \left\| {\left( {\sum\limits_{i \in
\gamma }^{} {\left| {b_i } \right|^2 } } \right)^{1/2} }
\right\|_{E(M,\tau)}  < \infty,$$
and therefore $\left\{{a_i + b_i} \right\}_{i \in I} \in E ( M,\tau,I )$
 and
$$ \|\{ a_i\}_{i\in I}\ + \{ b_i\}_{i\in I}\|_{E(M,\tau,I)}\leq\|\{
a_i\}_{i\in I}\ \|_{E(M,\tau,I)} +\|\ \{ b_i\}_{i\in
I}\|_{E(M,\tau,I)}.$$

\end{proof}

\begin{prop}
\label{pr_3_3}
If $E([0, \infty ])$ is a fully symmetric Banach function
space and $E([0,\infty ])$ has the Fatou property,  then
$(E(M,\tau,I), \|\cdot\|_{E(M,\tau,I)} )$ is a Banach space. In
addition, if $E([0,\infty ])$ has order continuous norm, then
$E_0(M,\tau,I)$ is dense in $E(M,\tau,I).$
\end{prop}
\begin{proof}Let $x_n=\{a_{i}^{(n)}\}_{i\in I}\in E(M,\tau,I)$ and $\|x_{m}-x_{n}\|_{E(M,\tau,I)}\rightarrow
0$ as $m,n\rightarrow\infty$ i.e. for any $\varepsilon
> 0$ there exists a number $n\left( \varepsilon \right)$, such
that for $n, m \ge n\left( \varepsilon \right)$ and $\gamma \in
\Gamma$ the inequality
$$\left\| {\left( {\sum\limits_{i \in \gamma }^{} {\left|
{a_i^{\left( m \right)}  - a_i^{\left( n \right)} } \right|^2 } }
\right)^{1/2} } \right\|_{E(M,\tau)}  < \varepsilon \eqno{(1)}
$$
holds. From the inequality
$$\left({\sum\limits_{i \in \gamma }^{}
{\left| {a_i^{\left( m \right)} -a_i^{\left( n \right)} }
\right|^2 } } \right)^{1/2} \ge \left( {\left| {a_i^{\left( m
\right)}  - a_i^{\left( n \right)} } \right|^2 } \right)^{1/2} =
\left| {a_i^{\left( m \right)}  - a_i^{\left( n \right)} }
\right|$$
[\cite{MCh}, Ch.2,Theorem 2.4.2] using (1)   we infer that
$\left\| {a_i^{\left( m \right)}  - a_i^{\left( n \right)} }
\right\|_{E(M,\tau)}  < \varepsilon $ for any $n, m \ge n\left(
\varepsilon \right)$,  $i \in \gamma $. Since the space $E (M,\tau)$ is
complete, there exists an element $a_i^{\left( 0 \right)} \in E
(M,\tau)$, such that $\left\| {a_i^{\left( n \right)} -
a_i^{\left( 0 \right)} } \right\|_{E(M,\tau)}  \to 0$ as  $n \to
\infty ,$ for any  $i \in I.$ Now let us show that $ x_0 = \left\{
{a_i^{\left( 0 \right)} } \right\}_{i \in I}^{}  \in E (M,\tau,I)
$ and $\left\| { x_n -  x_0 } \right\|_{E(M,\tau,I)} \to 0$ .
Since $a_i^{\left( n \right)} \buildrel {t_\tau } \over
\longrightarrow a_i^{\left( 0 \right)} $,  we  get $\left|
{a_i^{\left( m \right)} - a_i^{\left( n \right)} }
\right|\buildrel {t_\tau  } \over \longrightarrow \left|
{a_i^{\left( m \right)}  - a_i^{\left( 0 \right)} } \right|$  as
$n \to \infty $,  for every $i\in I$ (Theorem \ref{t_2_1}). Consequently,
$\sum\limits_{i \in \gamma }^{} {\left| {a_i^{\left( m \right)}  -
a_i^{\left( n \right)} } \right|^2 \buildrel {t_\tau } \over
\longrightarrow \sum\limits_{i \in \gamma }^{} {\left|
{a_i^{\left( m \right)}  - a_i^{\left( 0
\right)} } \right|^2 } } $   and by Theorem \ref{t_2_1} we have that \\
$\left( {\sum\limits_{i \in \gamma }^{} {\left| {a_i^{\left( m
\right)}  - a_i^{\left( n \right)} } \right|^2 } } \right)^{1/2}
\buildrel {t_\tau  } \over \longrightarrow \left( {\sum\limits_{i
\in \gamma }^{} {\left| {a_i^{\left( m \right)}  - a_i^{\left( 0
\right)} } \right|^2 } } \right)^{1/2} $ as $n \to \infty $. Since
the space $E(M,\tau)$ has the Fatou property we have that the unit
ball of $(E(M,\tau), \|\cdot\|_{E(M,\tau)})$ is closed in
$S(M,\tau)$ in the measure topology $t_\tau$ \cite{DDST}.
By (1) we obtain that
$$\|\left({\sum\limits_{i \in \gamma }^{} {\left| {a_i^{\left( m
\right)} -a_i^{\left( 0 \right)} } \right|^2 } } \right)^{1/2}
\|_{E(M,\tau)}\leq\varepsilon$$
 for all $m\geq n(\varepsilon),
\gamma\in\Gamma.$ Consequently,
$$\mathop {\sup }\limits_{\gamma \in \Gamma} \left\| {\left(
{\sum\limits_{i \in \gamma }^{} {\left| {a_i^{\left( m \right)}  -
a_i^{\left( 0 \right)} } \right|} ^2 } \right)^{1/2} }
\right\|_{E(M,\tau)} \le \varepsilon ,$$
 i.e. $\left\{ {a_i^{\left(m \right)}  - a_i^{\left( 0 \right)} } \right\} \in E
(M,\tau,I)$ and
$\left\| { x_m -  x_0 } \right\|_{E (M,\tau,I)} \le \varepsilon $
for $m \ge n\left( \varepsilon \right)$.  Hence $x_0 =  x_{_m } -
\left( { x_{_m } - x_0 } \right) \in E (M,\tau,I),$ and $\left\| {
x_{m } -
 x_{_0 } } \right\|_{E(M,\tau,I)} \to 0$ as  $m \to
\infty.$

We show now  that if  $E([0,\infty])$ has  order continuous norm,
then $E_0(M,\tau,I)$ is dense in $E(M,\tau,I).$  Let
$a=\{a_i\}_{i\in I}\in E(M,\tau,I)$ and
$a_\gamma=\{b^{(\gamma)}_i\}_{i\in I},$ where $b^{(\gamma)}_i=
a_i$ if $i\in \gamma,$  $b^{(\gamma)}_i= 0$ if $i\notin \gamma,$
$\gamma\in\Gamma.$ It is clear that $a_\gamma \in E_0(M,\tau,I)$
for all  $\gamma\in\Gamma.$ Fix an index $j_0,$ and consider the operator
$ T(a_\gamma)= \sum\limits_{i \in\gamma }^{} {a_i \otimes u_{ij_0
} }$ (see the proof of proposition \ref{pr_3_1}). Let $\gamma\leq\beta,
\gamma,\beta\in\Gamma.$ Since
$|T(a_\beta-a_\gamma)|=(\sum\limits_{i \in\
\beta\setminus\gamma}^{} {|a_i|^2)^{1/2} \otimes p_{j_0} }$ we
have that
$$\|a_\beta-a_\gamma\|_{E(M,\tau,I)}= \|T (a_\beta-a_\gamma)\|_{E
\left( M \otimes B\left( l_2 \left( I \right)\right)\right)}=
\|\sum\limits_{i \in\ \beta\setminus\gamma}^{} {|a_i|^2)^{1/2}}\|_{E(M,\tau)}.$$
The net  $0\leq z_\gamma:=(\sum\limits_{i \in\gamma}^{}
{|a_i|^2)^{1/2}} \in E(M,\tau) $ ($\gamma\in\Gamma$) is
increasing and $$\mathop {\sup }\limits_{\gamma \in
\Gamma} \left\| {z_\gamma } \right\|_{E(M,\tau)}= \mathop {\sup
}\limits_{\gamma \in \Gamma} \left\| {\left( {\sum\limits_{i \in
\gamma }^{} {\left| {a_i } \right|^2 } } \right)^{1/2} }
\right\|_{E(M,\tau)} =  \left\| a\right\|_{E(M,\tau, I)}<\infty.$$
Since $\left( {E(M,\tau),\left\| \cdot \right\|_{E (M,\tau)} }
\right)$ has the Fatou property and order continuous norm \cite
{DDST} it follows that there exists an operator $0\leq z \in E(M,\tau)$ such
that $z_\gamma \uparrow z$ and $\left\| {z - z_\gamma  }
\right\|_{E(M,\tau,I)} \downarrow 0.$  Consequently and $z_\gamma
\buildrel {t_\tau  } \over \longrightarrow z$ and  $z_\gamma ^2
\buildrel {t_\tau  } \over \longrightarrow z^2.$ Since
$z_\gamma^2=\sum\limits_{i \in\gamma}^{} {|a_i|^2}\leq
\sum\limits_{i \in\beta}^{} {|a_i|^2}= z_\beta$ for
$\gamma\leq\beta$, we have  $z_\gamma ^2 \uparrow z^2.$ Hence
$(z^2-z_\gamma ^2 )^{1/2}\buildrel {t_\tau  } \over
\longrightarrow 0$ (Theorem \ref{t_2_1}). Using the equalities $z^2-z_\gamma
^2 = \sum\limits_{i \in I}^{} {|a_i|^2}- \sum\limits_{i
\in\gamma}^{} {|a_i|^2}= \sum\limits_{I\setminus \gamma }^{}
{|a_i|^2}$ (the series converges in $t_\tau-$ topology) we
obtain that $(z^2-z_\gamma ^2 )^{1/2}\downarrow 0.$ The inequalities $0\leq
(z^2-z_\gamma ^2 )^{1/2}\leq z$ imply that $(z^2-z_\gamma ^2 )^{1/2}\in
E(M,\tau).$ Using the property of continuous norm
$\|\cdot\|_{E(M,\tau)}$  we get that  $\|(z^2-z_\gamma ^2
)^{1/2}\|_{E(M,\tau)}\downarrow 0.$ Consequently, for
$\gamma\leq\beta$ we have $$\|a_\beta-a_\gamma \|_{E(M,\tau, I)}=
\|(z_\beta^2-z_\gamma ^2 )^{1/2}\|_{E(M,\tau)}\leq \|(z^2-z_\gamma
^2 )^{1/2}\|_{E(M,\tau)}\downarrow 0.$$  This means that
$\{a_\gamma\}_{\gamma\in\Gamma}$ is a Cauchy net in the Banach space
$\left( {E(M,\tau,I),\left\| \cdot \right\|_{E (M,\tau,I)} }
\right).$ Hence there exists $c= \{c_i\}_{i\in I}\in E(M,\tau,
I)$ such that $\|c-a_\gamma \|_{E(M,\tau, I)}\rightarrow 0,$ in
particular, $\|b_i^{(\gamma)}-c_i \|_{E(M,\tau)}\rightarrow 0$ for
all $i\in I.$  This means that $c_i=a_i$ for every $i\in I,$ i.e.
$c=a.$ Thus $a= \left\|  \cdot \right\|_{E(M,\tau,I)} - \mathop
{\lim }\limits_{\gamma}
 {a_\gamma }.$ Consequently, $E_0(M,\tau,I)$ is dense in $E(M,\tau,I)$
\end{proof}

Following \cite {P} we call the constructed Banach space
$E(M,\tau,I)$ a non-commutative vector-valued fully symmetric
space, associated with fully symmetric Banach function space
$E([0,\infty)).$ In the case when $E([0,\infty))=
L_{p}([0,\infty)), p\geq1,$ the Banach spaces $L_{p}(M,\tau,I)$ are
investigated in \cite{P}.

 Let $E([0,\infty))$ be a fully symmetric
function space and let $E([0,\infty))$ has the Fatou property and order
continuous norm. Using the proof of Proposition \ref{pr_3_3} we define a
special linear isometry from $E(M,\tau, I)$ into $E(M \otimes
B(l_2(I))).$ Fix an index $j_0.$ Let $a= \{a_i\}_{i\in\Gamma}\in
E(M,\tau, I),$ $a_\gamma=\{b^{(\gamma)}_i\}_{i\in I}$ be as in the proof
of Proposition \ref{pr_3_3}. Since $\left\| { a -  a_\gamma }
\right\|_{E(M,\tau,I)} \to 0$ and $\left\| { a_\beta -  a_\gamma }
\right\|_{E(M,\tau,I)}=\|T (a_\beta)-T(a_\gamma)\|_{E \left( M
\otimes B\left( l_2 \left( I \right)\right)\right)}$, it follows that
$\{T(a_\gamma)\}_{\gamma\in \Gamma}$ is a Cauchy net in the Banach
space  $ (E \left( {M \otimes B\left( {l_2 \left( I
\right)}\mathbf{}\mathbb{} \right)} \right), \|\cdot\|_{E \left(
{M \otimes B\left( {l_2 \left( I \right)} \right)} \right)}).$
Consequently, there exists $\Phi(a)\in E \left( {M \otimes B\left(
{l_2 \left( I \right)} \right)} \right)$ such that $\|\Phi(a) - T(
a_\gamma)\|_{E \left( {M \otimes B\left( {l_2 \left( I \right)}
\right)} \right)}\rightarrow0.$ It is clear that if
$\{c_\gamma\}_{\gamma\in\Gamma}\in E_0(M,\tau,I)$ is an arbitrary
 net with $\|a-c_\gamma\|_{E(M,\tau,I)}\rightarrow 0,$ then $\|\cdot\|_{\left( {E(M \otimes B\left( {l_2 \left( I
\right)}\mathbf{}\mathbb{} \right)} \right))}- \mathop {\lim
}\limits_{\gamma}
 {T(c_\gamma)}= \Phi (a),$ i.e. map $\Phi$ is correctly defined.  From definition of $\Phi$ it follows that
 $\Phi$ is a linear map, in addition, $$\|\Phi (a)\|_{E \left( {M \otimes B\left( {l_2 \left( I \right)}
\right)} \right)} = \mathop {\lim
}\limits_{\gamma}\|T(a_\gamma)\|_{E \left( {M \otimes B\left( {l_2
\left( I \right)} \right)} \right)} = \mathop {\lim
}\limits_{\gamma}\|a_\gamma\|_{E(M,\tau,I)} =\|a\|_{E \left( M, \tau, I \right)} ,$$
i.e. $\Phi $ is an
isometric map. Since $\left( E (M,\tau,I), \left\|  \cdot
\right\|_{E (M,\tau,I)}\right) $ is a Banach space (Proposition
\ref{pr_3_3}), it follow that $\Phi \left( {E (M,\tau,I)} \right)$ is a closed linear
subspace of the Banach space $E \left( {M \otimes B\left( {l_2 \left( I \right)}
\right)} \right)$. Thus, we obtain the following corollary

\begin{cor}
\label{cor_3_4}
 $\Phi $ is a linear isometric map  from $\left( {E
(M,\tau,I), \left\| \cdot \right\|_{E (M,\tau,I)} } \right)$ onto a
  closed linear subspace in $E \left( {M \otimes B\left( {l_2
\left( I \right)} \right)} \right).$
\end{cor}

Now we show that $\Phi (E(M,\tau,I))$ is a complemented linear
subspace in the Banach space $E \left( {M \otimes B\left( {l_2 \left( I \right)}
\right)} \right).$
For a fixed $\gamma \in \Gamma$ we define the linear operator

$$P_\gamma  :E \left( {M \otimes B\left( {l_2 \left( I \right)}
\right)} \right) \to E \left( {M \otimes  B\left( {l_2 \left( I
\right)} \right)} \right)$$
by the equality
 $$P_\gamma
\left( x \right) = \sum\limits_{j \in \gamma }^{} {\left(
{\textbf{1} \otimes p_j } \right)} x\left( {\textbf{1} \otimes
p_{j_0 } } \right), x \in E \left( {M \otimes B\left( {l_2 \left(
I \right)} \right)} \right).$$

\begin{prop}
\label{pr_3_5}
 Let $\gamma, \beta  \in \Gamma, \gamma \leq \beta$. Then

 (i)  $P_\gamma $ is a bounded linear operator and $\left\|
{P_\gamma} \right\| \le 1$;

(ii)  $P_\gamma P_\beta   = P_\beta P_\gamma   = P_\gamma; $

(iii) $P_\gamma ^2  = P_\gamma  ;$

(iv) $P_\gamma  \left( {E \left( {M \otimes  B\left( {l_2 \left( I
\right)} \right)} \right)} \right) \subset \Phi \left( {E
(M,\tau,I)} \right).$
\end{prop}

\begin{proof}
(i) If $y=P_\gamma (x)$, then $$ y^ *  y = \left( {\sum\limits_{j
\in \gamma}^{} {\left( {\textbf{1} \otimes p_j } \right)} x\left(
{\textbf{1} \otimes p_{j_0 } } \right)} \right)^
*  \left( {\sum\limits_{j \in \gamma }^{} {\left( {\textbf{1} \otimes p_j }
\right)x\left( {\textbf{1} \otimes p_{j_0 } } \right)} } \right)
=$$

$$=\left( \sum\limits_{j \in \gamma }^{} \left( {\textbf{1 } \otimes p_{j_0 } }
\right)x^ *  \left( {\textbf{1} \otimes p_j } \right) \right)
\left( \sum\limits_{j \in \gamma }^{} \left( {\textbf{1} \otimes
p_j } \right)x\left( {\textbf{1} \otimes p_{j_0 } } \right)
\right)= $$
$$=\left( {\textbf{1} \otimes p_{j_0 } } \right)x^ *
\left( {\sum\limits_{j \in \gamma }^{} {\left( {\textbf{1} \otimes
p_j } \right)} } \right)x\left( {\textbf{1} \otimes p_{j_0 } }
\right).
$$

Since $\left\{ {\left( {\textbf{1} \otimes p_j } \right)}
\right\}_{j \in \gamma }^{}  $ is a family of pairwise orthogonal
projections in $M  \otimes  B\left( {l_2 \left( I \right)}
\right),$ we have $\sum\limits_{j \in \gamma }^{} {\left( {\textbf{1}
\otimes p_j } \right) \le \textbf{1}_{M \otimes B\left( {l_2
\left( I \right)} \right)} } .$

Consequently,  $$0 \le y^* y \le \left( {\textbf{1} \otimes p_{j_0
} } \right)x^
* \textbf{1}_{M \otimes B\left( {l_2 \left( I \right)} \right)}
x\left( {\textbf{1} \otimes p_{j_0 } } \right) = \left(
{\textbf{1} \otimes p_{i_0 } } \right)x^ *  x\left( {\textbf{1}
\otimes p_{i_0 } } \right)=$$ $$ = \left( {x\left( {\textbf{1}
\otimes p_{j_0 } } \right)} \right)^* \left( {x\left( {\textbf{1}
\otimes p_{j_0 } } \right)} \right),$$ i.e. $\left| y \right|^2
\le \left| {x\left( {\textbf{1} \otimes p_{j_0 } } \right)}
\right|^2 .$ By (\cite{BChS}) there exists $c \in
M \otimes B\left( {l_2 \left( I \right)} \right)$ such that
$\left\| c \right\|_{M \otimes B\left( {l_2 \left( I \right)}
\right)}  \le 1$ and $\left| y \right| = c^
*  \left| {x\left( {\textbf{1} \otimes p_{j_0 } } \right)}
\right|c.$ Hence,

$ \mu _t \left( {\left| y \right| } \right) =  \mu _t \left( {c^ *
\left| {x\left( {\textbf{1} \otimes p_{j_0 } } \right)} \right|c}
\right)  \le \left\| c \right\|_{M  \otimes B\left( {l_2 \left( I
\right)} \right)}^2 \mu _t \left( {\left| {x\left( {\textbf{1}
\otimes p_{j_0 } } \right)} \right| } \right) \le \mu _t \left(
{\left| {x\left( {\textbf{1} \otimes p_{j_0 } } \right)} \right| }
\right).$ Thus,

 $$\|P_\gamma (x)\|_{E \left( {M \otimes B\left(
{l_2 \left( I \right)} \right)} \right)}\leq\|{x\left( {\textbf{1}
\otimes p_{j_0 } } \right)}\|_{E \left( {M \otimes B\left( {l_2
\left( I \right)} \right)} \right)}\leq$$

$$\left\| {\textbf{1} \otimes p_{j_0 } } \right\|_{M
\otimes B\left( {l_2 \left( I \right)} \right)}  \left\| x
\right\|_{E \left( {M \otimes B\left( {l_2 \left( I \right)}
\right)} \right)}  \le \left\| x \right\|_{E \left( {M \otimes
B\left( {l_2 \left( I \right)} \right)} \right)}.$$

Statements  (ii),(iii) follow from definition.

(iv). If $x = \sum\limits_{i=1}^n  {a_i  \otimes b_i \in{\left( {M
\otimes B\left( {l_2 \left( I \right)} \right)} \right)}\cap E
\left( {M \otimes B\left( {l_2 \left( I \right)} \right)} \right)}
$, then

$$P_\gamma \left( x \right) = \sum\limits_{j \in \gamma }^{}
{\left( {\textbf{1} \otimes p_j } \right)\left( {\sum\limits_{i
=1}^n {a_i  \otimes b_i } } \right)\left( {\textbf{1} \otimes
p_{j_0 } } \right) = \sum\limits_{j \in \gamma }^{}
{\sum\limits_{i =1 }^n {\left( {a_i  \otimes \left( {p_j b_i
p_{j_0 } } \right)} \right).} } } $$
By Proposition \ref{pr_2_3} we have
that $$P_\gamma \left( x \right)= \sum\limits_{j \in\gamma }^{}
{\sum\limits_{i =1 }^n {(b_i e_{j_0},e_j)\left(  a_i \otimes u_{jj_0
} \right)} }\in \Phi \left( {E \left( {M,\tau,I} \right)}
\right).$$
 Since $E \left( {M \otimes B\left( {l_2 \left( I
\right)} \right)} \right)$ has order continuous norm \cite{DDST}
we have that
$${\left( {M \otimes B\left( {l_2 \left( I \right)} \right)}
\right)}\cap E \left( {M \otimes B\left( {l_2 \left( I \right)}
\right)} \right) $$    is dense in $ E \left( {M \otimes B\left(
{l_2 \left( I \right)} \right)} \right)$ \cite {DDST}.
Consequently for     $x\in E \left( {M \otimes B\left( {l_2 \left(
I \right)} \right)} \right)$ there exists a sequence
$\{x_n\}\subset {\left( {M \otimes B\left( {l_2 \left( I \right)}
\right)} \right)}\cap E \left( {M \otimes B\left( {l_2 \left( I
\right)} \right)} \right) $ such that
$$\|x_n-x\|_{E \left( {M \otimes B\left( {l_2 \left( I \right)}
\right)} \right)}\rightarrow 0.$$

By (i) we have that $\left\|P_\gamma \left( {x_n } \right) - P_\gamma
( x ) \right\|_{E \left( {M \otimes B\left( {l_2 \left( I \right)}
\right)} \right)} \to 0$.  Since $P_\gamma \left( {x_n } \right)
\in \Phi (E (M,\tau,I))$ and $ \Phi (E (M,\tau,I))$ is a closed
set, it follows that  $P_\gamma \left( x \right) \in \Phi (E (M,\tau,I)),$
i.e. $$P_\gamma \left( {E \left( {M \otimes B\left( {l_2 } \right)}
\right)} \right) \subset \Phi (E (M,\tau,I)).$$

\end{proof}

Now we  show that $\left\{ {P_\gamma  \left( x \right)}
\right\}_{\gamma  \in \Gamma} $ \ \  $\left\|  \cdot \right\|_{E
\left( {M \otimes B\left( {l_2 \left( I \right)} \right)}
\right)}$
-converges  for any \\
$x \in E \left( {M \otimes B\left( {l_2 \left( I \right)} \right)}
\right)$.
 If $a \in E(M,\tau),b\in E(B(l_2(I)))$, then by Proposition \ref{pr_2_3}, we have
$$P_\beta \left( {a \otimes b} \right) = \sum\limits_{j \in \beta }^{} {a \otimes p_{j}bp_{j_0 }
=  {\sum\limits_{j \in \beta }^{} {(be_{j_0},e_j)\left( a \otimes u_{jj_0 } \right)} }  =
\sum\limits_{j \in \beta }^{} {a_{j} \otimes u_{jj_0 }}},$$
where $a_j  = (be_{j_0},e_j)a$, $j \in
\beta $. Thus,
$$\left\| {P_\beta  \left( {a \otimes b} \right) - P_\gamma
\left( {a \otimes b} \right)} \right\|_{E \left( {M \otimes
B\left( {l_2 \left( I \right)} \right)} \right)}  =
\left\| {\left( {\sum\limits_{j \in \beta \backslash \gamma }^{}
{\left| {a_j } \right|^2 } } \right)^{1/2} }
\right\|_{E(M,\tau,I)} =$$
$$\| {\left( {\sum\limits_{j \in \beta
\backslash \gamma } {\left| (be_{j_0},e_j)\right| } }
\right)^{1/2} \left| a \right| }\|_{E(M,\tau,I)}=
\left( {\sum\limits_{j \in \beta \backslash \gamma } {\left|
(be_{j_0},e_j)\right| } } \right)^{1/2}\|a\|_{E(M,\tau,I)}$$
for all $\beta \ge \gamma  \ge \left\{ {j_0 } \right\}$.

Since $\left\| {be _{j_0 } } \right\|_{B(l_2(I))} =
\sum\limits_{j \in I} {\left| {\left( {be_{j_0 } ,e_j }
\right)} \right|^2 } $ we have that  $\{P_\gamma  \left( {a
\otimes b} \right)\}$ is a
 Cauchy net in $E(M \otimes B(l_2 ( I)))$ and it
 $\left\|\cdot \right\|_{E \left( {M \otimes B\left( {l_2 \left( I \right)}
\right)} \right)}  - $ converges in $\Phi \left( {E (M,\tau,I)}
\right).$ Its  $\| \cdot \|_{E \left( {M \otimes B\left( {l_2
\left( I \right)} \right)} \right)} $ -limit we denote by $P\left(
{a \otimes b} \right).$ Since $P_\gamma $ is a linear map we have that
for every $a_i \in E \left( M,\tau \right), b_i \in E \left( {M
\otimes B\left( {l_2 \left( I \right)} \right)} \right), i =
{1,...,n} $, the net $P_\gamma \left( {\sum\limits_{i = 1}^n {a_i
\otimes b_i } } \right)$ $\left\| \cdot  \right\|_{E \left( {M
\otimes B\left( {l_2 \left( I \right)} \right)} \right)} - $
converges to the element $\sum\limits_{i = 1}^n {P\left( {a_i \otimes
b_i } \right)} \in \Phi \left( {L_p (M,\tau,I)} \right) $. Since
$E \left( {M \otimes B\left( {l_2 \left( I \right)} \right)}
\right)$ has order continuous norm, it follows that the set  $E \left( {M )\otimes
E(B\left( {l_2 \left( I \right)} \right)} \right)$ is dense in $E
\left( {M \otimes B\left( {l_2 \left( I \right)} \right)}
\right).$  Using the inequality $\left\| {P_\gamma  } \right\| \le
1$ for all $\gamma \in \Gamma$, we obtain that $P_\gamma \left( x
\right)$ \ \ $\left\| \cdot \right\|_{E \left( {M \otimes B\left(
{l_2 \left( I \right)} \right)} \right)} - $ converges for every
$x \in E \left( {M \otimes B\left( {l_2 \left( I \right)} \right)}
\right)$ to some element which we denote  by $P_{} \left( x
\right)$. It is clear that $P$ is a linear map from $E\left(
{M \otimes B\left( {l_2 \left( I \right)} \right)} \right)$ into
$\Phi (E (M,\tau,I))$.

\begin{prop}
\label{pr_3_6}
\item (i) $\left\| P \right\| \le 1$; \item (ii)  $P^2  = P$;
\item (iii) $P\left( {E \left( {M  \otimes  B\left( {l_2 \left( I
\right)} \right)} \right)} \right) = \Phi (E (M,\tau,I))$.

\begin{proof}
 The statement (i) follows immediately from Proposition \ref{pr_3_5}(i).
\item (ii).  Since
$P_\beta \left( {P_\gamma  x} \right) = P_\gamma \left( x
\right)$ for $\beta  \ge \gamma $, we have that $P \left( {P_\gamma \left(
x \right)} \right) = P_\gamma \left( x \right)$ for every $\gamma
\in\Gamma.$ Then, the following convergence $\left\| P_\gamma(x)-P(x) \right\|_{E \left( {M
\otimes B\left( {l_2 \left( I \right)} \right)}
\right)}\rightarrow 0 $ implies that $P^2  = P$.  \item (iii). Let $x =
\sum\limits_{i \in I}^{} {a_i \otimes u_{ij_0 } } $ be an
arbitrary element in $\Phi (E(M,\tau,I))$ (the series
converges in $E \left( {M \otimes B\left( {l_2 \left( I \right)}
\right)} \right))$ and let $x_\gamma = \sum\limits_{i \in \gamma
}^{} {a_i \otimes u_{ij_0 } }. $
Since $\|x-x_\gamma\|_{E \left(
{M \otimes B\left( {l_2 \left( I \right)} \right)} \right)}\to 0$,
$P_\gamma \left( x \right) = x_\gamma $ and $\left\|
P_\gamma(x)-P(x) \right\|_{E \left( {M \otimes B\left( {l_2 \left(
I \right)} \right)} \right)}\rightarrow 0 $ we have $P\left( x
\right) = x.$ Hence, $$ \Phi (E (M,\tau,I))=
 P\left( \Phi (E (M,\tau,I) \right)) \subset P\left({E \left( {M
\otimes B\left( {l_2 \left( I \right)} \right)} \right)} \right)
\subset \Phi (E (M,\tau,I)).$$
\end{proof}
\end{prop}

Now we are ready to prove of theorem \ref{t_1_1}.

\begin{proof}
(i). Let
 $$\widetilde E_i = E_i\left( {M_i \otimes B\left( {l_2
\left( I \right)} \right)} \right), \widetilde F_i=
F_i(M_i\otimes B(l_2(I))), \widetilde E  = E\left( {M_1
\otimes B\left( {l_2 \left( I \right)} \right)} \right),$$
$$\widetilde F= F(M_2\otimes B(l_2(I))),
\widehat{E_i}=E_i(M_i,\tau_i,I),
\widehat{F_i}=F_i(M_i,\tau_i,I),$$
$$\widehat{E}=E(M_1,\tau_1,I), \widehat{F}=F(M_2,\tau_2,I),i=1,2.$$
Denote by $\Phi_i$ the linear
isometrical map from $(L_1+L_\infty)(M_i,\tau_i,I)$ into
$(L_1+L_\infty)(M_i\otimes B(l_2(I)))= L_1(M_i\otimes B(l_2(I)))+
M\otimes B(l_2(I))$ (see Corollary \ref{cor_3_4}), $i=1,2.$ The construction
of $\Phi_i$ (see the proof of Propositions \ref{pr_3_1} and \ref{pr_3_3}) implies that
$\Phi_i$ is an isometrical map from  $\widehat{E_i}$ (respectively, $\widehat{F_i}, \widehat{E}$) into $\widetilde {E_i}$ (respectively, $\widetilde {F_i}, \widetilde {F}$).

Let $$P:(L_1+L_\infty)(M_1\otimes B(l_2(I)))\rightarrow
\Phi_1((L_1+L_\infty)(M_1,\tau_1,I)) $$ be the projection from Proposition \ref{pr_3_6}. Using constructions of
the isometry $\Phi_1$ and the projection $P$, we have that
$$P(\widetilde
{E_1})=\Phi_1(\widehat{E_1}), P(\widetilde {F_1})=\Phi_1(
\widehat{F_1}),$$
$$P(\widetilde {E_1}+ \widetilde
{F_1})=\Phi_1(\widehat{E_1}+ \widehat{F_1}), \ \
 P(\widetilde {E_1}\cap
\widetilde {F_1})=\Phi_1(\widehat{E_1} \cap \widehat{F_1}).$$

Let
$T:\widehat E_1+ \widehat F_1\rightarrow \widehat{E_2}+
\widehat{F_2}$ be a linear operator such that $T$ is bounded
operator from $\widehat{E_1}$ into $\widehat{E_2},$
$\widehat{F_1}$ into $\widehat{F_2}$ respectively. We define a
linear operator
$$\widetilde T:\widetilde E_1+ \widetilde
F_1\rightarrow \widetilde{E_2}+ \widetilde{F_2}$$
by
$$\widetilde
T(x+y)=(\Phi_2T\Phi_1^{-1}P)(x+y), \ \ x\in\widetilde E_1,
 y\in\widetilde F_1.$$
 It is clear that
 $$\widetilde T(\widetilde
{E_1})\subset\widetilde {E_2}, \ \   \widetilde T(\widetilde
{F_1})\subset\widetilde {F_2},$$
$$\|\widetilde T\|_{\widetilde {E_1}\rightarrow \widetilde
E_2}\leq\|\Phi_2\| \|T\|_{\widehat {E_1}\rightarrow \widehat
{E_2}}\|\Phi_1^{-1}\|\|P\|\leq\|T\|_{\widehat {E_1}\rightarrow
\widehat{E_2}}, $$
and similarly
$\|\widetilde T\|_{\widetilde F_1\rightarrow \widetilde F_2}\leq\|
T\|_{\widehat {F_1}\rightarrow \widehat{F_2}}$. By
Theorem \ref{t_2_5}(i) we obtain that  $\widetilde T(\widetilde
E)\subset\widetilde F,$ and
$$\|\widetilde T\|_{\widetilde E
\rightarrow \widetilde F}\leq c \max(\| T\|_{ \widehat
E_1\rightarrow \widehat E_2},\|T\|_{ \widehat F_1\rightarrow
\widehat F_2})$$ for some $c>0.$ Since $P^2=P,$ $P(\widetilde
E)=\Phi_1(\widehat E)$ (see Proposition \ref{pr_3_6} (ii),(iii)), it follows that for
every $z\in\widehat E$ we have $P(\Phi_1(z))=\Phi_1(z)$ and
$$\widehat{T}(\Phi_1(z))=(\Phi_2T\Phi_1^{-1}P)(\Phi_1(z))=\Phi_2(T(z))\in\widehat F.$$
This means that $T(z)=\Phi_2^{-1}(\widetilde
T(\Phi_1(z)))\in\widehat F.$ Hence $T(\widehat E)\subset\widehat
F$ and $$\| T\|_{\widehat E\rightarrow \widehat F}\leq\|\widetilde
T\|_{\widetilde E\rightarrow \widetilde F}\leq
c\|T\|_{B((\widehat{E_1},\widehat{F_1}),(\widehat{E_2},\widehat{F_2}))}.$$
Thus,  $(\widehat E,\widehat F)$ is a interpolation pair for the
pair $(((\widehat{E_1},\widehat{F_1}),(\widehat{E_2},\widehat{F_2}))).$
If $(E,F)$ is an exact interpolation pair for pair
$((E_1,F_1),(E_2,F_2))$ then $c=1.$ Hence
$(\widehat{E},\widehat{F})$ is an exact interpolation pair for
$(((\widehat{E_1},\widehat{F_1}),(\widehat{E_2},\widehat{F_2}))).$
The proof of statement \item(ii) is similar to that of (i).

\end{proof}

\end{document}